\documentclass[a4paper,11pt,twoside,leqno]{article}
\usepackage{amsmath,amssymb,amsfonts,amsthm,graphicx,fancyhdr}
\usepackage[latin1]{inputenc}
\usepackage[T1]{fontenc}
\usepackage{rotating}
\usepackage[colorlinks=true]{hyperref}
\usepackage{suetterl}
\usepackage{thmtools}

\newtheorem{teo}{Theorem}[subsection]
\newtheorem{lem}[teo]{Lemma}
\newtheorem{prop}[teo]{Proposition}
\newtheorem{cor}[teo]{Corollary}
\newtheorem{dfn}[teo]{Definition}
\newtheorem{ques}[teo]{Question}

\newtheorem*{corn}{Corollary}

\newtheorem*{tea}{Theorem}

\declaretheoremstyle[
  spaceabove=\topsep, spacebelow=\topsep,
  headfont=\bf,  
  notefont=\mdseries, notebraces={(}{)},
  bodyfont=\rmfamily, 
  postheadspace=1em,
  qed=$\Diamond$
]{drem}
\declaretheorem[style=drem, name=Remark, numberlike=teo]{rmk}
\declaretheorem[style=drem, name=Example, numberlike=teo]{exa}

\newcommand{\eg}[0]{\emph{e.g.} }
\newcommand{\ie}[0]{\emph{i.e.} }

\newcommand{\ssl}[1]{\underline{#1}}
\newcommand{\srl}[1]{\overline{#1}}
\newcommand{\jo}[1]{\mathcal{#1}}
\newcommand{\id}[1]{\mathfrak{#1}}

\newcommand{\pgen}[1]{\langle #1 \rangle}

\newcommand{\imp}[0]{\Rightarrow}

\newcommand{\eps}[0]{\varepsilon}

\newcommand{\kk}[0]{\ensuremath{\mathbb{K}}}
\newcommand{\rr}[0]{\ensuremath{\mathbb{R}}}
\newcommand{\zz}[0]{\ensuremath{\mathbb{Z}}}
\newcommand{\nn}[0]{\ensuremath{\mathbb{N}}}

\newcommand{\del}[0]{\ensuremath{\partial}}
\newcommand{\dd}[0]{\ensuremath{\mathrm{d}\!}}

\newcommand{\Id}[0]{\mathrm{Id}}

\newcommand{\inj}[0]{\hookrightarrow}

\newcommand{\limm}[1]{\textrm{\raisebox{.5ex}{\mbox{$\underset{#1}{\lim}$}}} \:}

\newcommand{\supp}[1]{\textrm{\raisebox{.5ex}{\mbox{$\underset{#1}{\sup}$}}} \:}
\newcommand{\inff}[1]{\textrm{\raisebox{.5ex}{\mbox{$\underset{#1}{\inf}$}}} \:}

\newcommand{\vide}[0]{\varnothing}

\newcommand{\comp}[0]{\mathsf{c}}

\newcommand{\eqtag}[0]{\addtocounter{teo}{1} \tag{\theteo}}

\newcommand{\cay}[0]{\mathrm{Cay}\!}

\newcommand{\IS}[0]{\mathrm{I}\mathrm{S}}

\newcommand{\FP}[0]{\mathrm{F}\!\mathrm{P}\!}

\newcommand{\D}[0]{\mathsf{D}}
\newcommand{\LL}[0]{\mathsf{L}}

\newcommand{\maza}[0]{\mu}

\newcommand{\Gver}[0]{X}
\newcommand{\Grou}[0]{G}

\newcommand{\tvantran}[0]{\cite[Theorem 1.3]{moi-trans}}
\newcommand{\tcentrinf}[0]{\cite[Theorem 3.2]{moi-trans}}

\newcommand{\EN}[0]{\text{\upshape \LARGE \suetterlin e}}
\newcommand{\ds}[1]{%
  \ifmmode
    \mathchoice{\text{\upshape \LARGE\suetterlin #1}\,}{\text{\upshape\LARGE\suetterlin #1}\,}{\text{\upshape\large\suetterlin #1}\,}{\text{\upshape \scriptsize\suetterlin #1}\,}%
  \else
    {\suetterlin #1}
  \fi
}

\DeclareSymbolFont{deusch}{T1}{suetterl}{m}{n}
\DeclareSymbolFontAlphabet{\mathds}{deusch}

\DeclareFontFamily{T1}{mafra}{}
\DeclareFontShape{T1}{mafra}{m}{n}{<->s*[0.95]yswab}{}
\DeclareFontShape{T1}{mafra}{m}{it}{<->s*[1.0]ygoth}{}
\DeclareTextFontCommand{\textgoth}{\yfrak}
\DeclareSymbolFont{mafrak}{T1}{mafra}{m}{n}
\DeclareSymbolFontAlphabet{\mathfr}{mafrak}
\renewcommand{\id}[1]{\mathfr{#1}}

\begin{document}
\centerline{\Large Boundary values, random walks and $\ell^p$-cohomology in degree one\footnote{MSC: Primary 20J06; Secondary: 05C81, 31C05, 60J45, 60J50}}

\vspace*{1cm}

\centerline{\large Antoine Gournay\footnote{Université de Neuchâtel, Rue É.-Argand 11, 2000 Neuchâtel, Suisse.}} 

\vspace*{1cm}

\centerline{\textsc{Abstract}}

\begin{center}
\parbox{10cm}{{ \small 
\hspace*{.1ex} The vanishing of reduced $\ell^2$-cohomology for amenable groups can be traced to the work of Cheeger \& Gromov in \cite{CG}. The subject matter here is reduced $\ell^p$-cohomology for $p \in ]1,\infty[$, particularly its vanishing. Results for the triviality of $\ssl{\ell^pH}^1(G)$ are obtained, for example: when $p \in ]1,2]$ and $G$ is amenable; when $p \in ]1,\infty[$ and $G$ is Liouville (\eg of intermediate growth).

\hspace*{.1ex} This is done by answering a question of Pansu in \cite[\S{}1.9]{Pan-rs} for graphs satisfying certain isoperimetric profile. Namely, the triviality of the reduced $\ell^p$-cohomology is equivalent to the absence of non-constant harmonic functions with gradient in $\ell^q$ ($q$ depends on the profile). 
In particular, one reduces questions of non-linear analysis ($p$-harmonic functions) to linear ones (harmonic functions with a very restrictive growth condition).
}}
\end{center}

\section{Introduction}\label{s-intro}

\setcounter{teo}{0}
\renewcommand{\theteo}{\thesection.\arabic{teo}}
\renewcommand{\theques}{\thesection.\arabic{teo}}
\renewcommand{\thecor}{\thesection.\arabic{teo}}

A graph $\Gamma = (\Gver,E)$ is defined by $\Gver$, its set of vertices, and $E$, its set of edges. All graphs will be assumed to be of bounded valency. The set of edges will be thought of as a subset of $\Gver \times \Gver$. The subject matter is the reduced $\ell^p$-cohomology in degree one of the graph $\Gamma$. This is the quotient 
\[
\ssl{\ell^p H}^1(\Gamma):= \D^p(\Gamma) / \srl{ \ell^p(\Gver) + \kk}^{\D^p},
\]
where $\kk = \rr$ or $\mathbb{C}$ is the field where our functions take values (this choice plays no role). See subsection \S{}\ref{ss-prelim} for more details. The main goal of this paper is to give partial answers to a question (dating back at least to Gromov \cite[\S{}8.$A_1$.($A_2$), p.226]{Gro}):
\begin{ques}\label{laquestion}
Let $G$ be an amenable group, is it true that for one (and hence all) Cayley graph $\Gamma$ and all $1<p<\infty$, $\ssl{\ell^pH}^1(\Gamma)=0$?
\end{ques}
The original question concerns cohomology in all degrees.
The case $p=1$ is slightly singular and the case $p=\infty$ is trivially false (see appendix \ref{sapp}). For $p=2$, the positive answer is a famous result of Cheeger \& Gromov \cite{CG} (see also Lück's book \cite{Luck}). The results presented here will give a positive answer for all $p \in ]1,2]$ and covers (for all $p$) many cases.

The basic idea relies on a standard argument which shows that a function which (essentially) only takes one value at infinity has trivial cohomology class. With a few more efforts, this determines $\ssl{\ell^1H}^1$ via boundary values on the end (see appendix \ref{sapp}). The idea is to define a ``boundary value'' of $g \in \D^p(\Gamma)$ on another ideal boundary, namely the Poisson boundary. This boundary is made up by harmonic functions, so a natural candidate for this boundary value is $\limm{n \to \infty} P^{(n)}g$ where $P$ is the random walk operator. 

The convergence of this limit can be obtained as a consequence of return probability (or heat kernels) estimates. To see that the behaviour of the boundary value still says something about the behaviour of $g$ at infinity, a transport problem (between a Dirac measure and the time $n$ distribution of a simple random walk) has to be studied. Some hypothesis on the isoperimetric profile will be necessary. For $F \subset X$ a subset of the vertices, let $\del F$ be the edges between $F$ and $F^\comp$. Let $d \in \rr_{\geq 1}$. Then, a graph $\Gamma$ has 
\[
\begin{array}{lllr}
\IS_d & \text{if there is a }  \kappa >0 \text{ such that for all finite } F \subset X, & |F|^{(d-1)/d} &\leq \kappa |\del F|; \\
\IS_\omega & \text{if there is a }  \kappa >0 \text{ such that for all finite } F \subset X, & |F|      &\leq \kappa |\del F|. 
\end{array}                                                                                                                                 \]
Quasi-homogeneous graphs with a certain (uniformly bounded below) volume growth in $n^d$ will satisfy these isoperimetric profiles, see Woess' book \cite[(4.18) Theorem]{Woe}. The Cayley graph of a group $G$ does not satisfy $\IS_\omega$ if and only if $G$ is amenable. It satisfies $\IS_d$ \emph{for all} $d$ if and only if $G$ is not virtually nilpotent. The upcoming result will apply best to groups which are not virtually nilpotent. See \S{}\ref{ss-prelim} or \cite[\S{}14]{Woe} for more details.
\begin{teo}\label{tpoisson}
Let $\Gamma$ be a graph satisfying $\IS_d$ and let $1\leq q \leq p <d/2$. Then 
\begin{itemize}
\item the natural quotient $\ssl{\ell^qH}^1(\Gamma) \to \ssl{\ell^pH}^1(\Gamma)$ is an injection;
\item if $\Gamma$ has no non-constant bounded harmonic functions whose gradient is in $\ell^p(E)$ then $\forall q <\tfrac{pd}{p+2d}, \ssl{\ell^qH}^1(\Gamma) = \{0\}$;
\item if $\Gamma$ has a non-constant (bounded or not) harmonic functions whose gradient is in $\ell^p(E)$ then $\ssl{\ell^pH}^1(\Gamma) \neq \{0\}$. 
\end{itemize}
\end{teo}
More precisely, a map from $\D^p(\Gamma)$ to harmonic functions modulo constants on $\Gamma$ is exhibited, and it is shown does not depends on the representative of the reduced $\ell^p$-cohomology class ($1 \leq p < \infty$). This maps sends bounded functions to bounded functions. To establish vanishing of reduced $\ell^p$-cohomology in degree $1$, it is sufficient to consider only bounded functions (for $1<p<\infty$), thanks to a lemma of Holopainen \& Soardi \cite[Lemma 4.4]{HS} (see Lemma \ref{tholoso-l} for a simple proof).

Theorem \ref{tpoisson} almost answers a question of Pansu \cite[Question 6 in \S{}1.9]{Pan-rs}: if $\Gamma$ has $\IS_d$ for all $d$, is the existence of a non-constant harmonic form whose gradient is in $\ell^p(E)$ equivalent to non-trivial reduced $\ell^p$-cohomology in degree $1$? Theorem \ref{tpoisson} shows this holds if one allows to lose some regularity ($q$ is bigger than $p$).
This theorem is the compilation of Corollaries \ref{tpoiss1-c}, \ref{tpoiss2-c} and \ref{tpoiss3-c}.

Note that though the current methods are not sufficient to show the existence of a harmonic function with finite $\ell^p$ gradient in each reduced cohomology class, it is an easy consequence of the methods that if such a function exists (and the isoperimetric profile is nice enough), then it is unique up to a constant, see Remark \ref{runiq}. 

Recall that all groups of subexponential growth are Liouville (see Avez \cite{Avez74}), \ie the Poisson boundary associated to the simple random walk on the Cayley graph is trivial (in this article, by ``trivial Poisson boundary'', one should always read trivial Poisson boundary for simple random walk in the Cayley graph).
\begin{corn}
If $G$ is a group of growth at least polynomial of degree $d$ and one of its Cayley graphs has trivial Poisson boundary (for the simple random walk, \ie $G$ is Liouville), then $\ssl{\ell^pH}^1(\Gamma) =\{0\}$ for any $1 \leq p < d/2$. 

In particular, groups of intermediate growth and $\zz_2 \wr \zz$ has trivial reduced $\ell^p$-cohomology in degree $1$, for any $p \in [1,\infty[$.
\end{corn}
Note it is unknown whether being Liouville is an invariant of quasi-isometry for Cayley graphs (it is not even known whether it is possible for a group to have a non-Liouville Cayley graph and a Liouville one). But, for the current purposes, it actually suffices that $G$ has a Cayley graph quasi-isometric to a Liouville graph in order to have that its $\ell^p$-cohomology vanish for all $p \in ]1,\infty[$.

Using \tvantran~and Theorem \ref{tpoisson} one can show that [amenable] lamplighters on $\zz^d$ (\eg $\zz_2 \wr \zz^d$) have harmonic functions with gradient in $\ell^p$. See also \cite{moi-lpharm} for more results on semi-direct products $G = N \rtimes H$ where $N$ is not finitely generated but $G$ is.

The reduced $\ell^2$-cohomology is trivial by Cheeger \& Gromov \cite{CG} (and the $\ell^p$-cohomology is trivial for any $p \in ]1,\infty[$ for groups which \emph{do not} have $\IS_d$ for some $d$, see \S{}\ref{ss-discu} below), using Theorem \ref{tpoisson} one gets a positive answer to question \ref{laquestion}:
\begin{corn}
\emph{Any} finitely generated amenable groups has trivial reduced $\ell^p$-cohomology (in degree one) for all $p \in ]1,2]$. Virtually-$\zz$ groups are the only amenable groups with non-trivial reduced $\ell^1$-cohomology in degree $1$.
\end{corn}

The Poisson boundary is not an invariant of quasi-isometry (see, for example, T.~Lyons' examples \cite{Lyo}). However, the following corollary, which was known for $p=2$ (trivially), may now be extended:
\begin{corn}
If $\Gamma$ is a graph satisfying $\IS_d$ and $\ssl{\ell^pH}^1(\Gamma) \neq \{0\}$ for some $1 \leq p<d/2$, then any graph quasi-isometric to $\Gamma$ has non-trivial Poisson boundary.
\end{corn}
Indeed, if $\Gamma$ satisfies $\IS_d$ and given that $p< d/2$, then inside the Poisson boundary lives the image (by boundary values) of the non-trivial reduced $\ell^p$-cohomology class representable by bounded functions; denote this image $\jo{P}_p$. Furthermore, if $D = \supp{\Gamma \text{ has } \IS_d} d \in [1,\infty]$, then $\jo{P} = \cup_{p < 2D} \jo{P}_p$ is a part of the Poisson boundary which will persist under quasi-isometry. One could also try to thicken $\jo{P}$ by considering more generic Banach spaces (\eg Orlicz spaces).

The second result is but a consequence of the first.
\begin{teo}\label{tinclu-t}
If $\Gamma$ has $\IS_\omega$ with constant $\kappa$, then the cohomology is always reduced (\ie $\ssl{\ell^p H}^1(\Gamma) = \ell^p H^1(\Gamma)$). Furthermore, if $n = \lceil \kappa^{-1} \rceil$ and $\Gamma^{[n]}$ is the $n$-fuzz of $\Gamma$ (the graph obtained from $\Gamma$ by adding edges between all points at distance $n$ in $\Gamma$), then there is a spanning tree $T$ in $\Gamma^{[n]}$ so that the non-trivial cohomology classes are exactly those non-trivial cohomology class in $T$ which also belong to $\D^p(\Gamma^{[n]})$.
\end{teo}
The equality between reduced and unreduced cohomology was already known. For Cayley graphs, see Guichardet \cite[Corollaire 1]{Guich} or Martin \& Valette \cite[Corollary 2.4]{MV}): a group is amenable if and only if $\ssl{\ell^p H}^1(G) \neq \ell^p H^1(G)$, for some, and hence any, $p \in ]1,\infty[$. For more general graphs, it is implicit at least in Lohoué \cite{Loh}.

In the author's mind, the interest of this results lies in the following idea: to compute the $\ell^p$ cohomology of a graph with positive isoperimetric constant, one needs only to run through the list of boundary values for a spanning tree, and look which one are in $\D^p$ of the initial graph. Of course, even if boundary values of the tree are somehow much more reasonable to compute (either by the methods of Bourdon \& Pajot in \cite{BP} or as the harmonic functions associated to the random walk), this is probably not directly usable unless the spanning tree produced by Benjamini \& Schramm in \cite{BS} may be made explicit.

Some results on $\ell^{p,q}$-cohomology are presented in \S{}\ref{ss-varia}.

Using Theorem \ref{tpoisson}, it is possible to get a vanishing result for groups with normal subgroups:
\begin{teo}\label{tvanquot-t}
Let $p\in [1,\infty[$. Assume $G$ is a finitely generated group, $N \lhd G$ is finitely generated as a group and the growth of $N$ is at least polynomial of degree $>2p$. Assume further that $G/N$ is infinite and $\ssl{\ell^pH}^1(\Gamma_N) =\{0\}$, where $\Gamma_N$ is some Cayley graph of $N$. Then for all Cayley graphs $\Gamma_G$ of $G$, $\ssl{\ell^pH}^1(\Gamma_G) =0$. If further $N$ is non-amenable then the statement is true in unreduced cohomology.
\end{teo}
The case where $N$ is non-amenable was already done in Bourdon, Martin \& Valette \cite[Theorem 1,1)]{BMV}. Bourdon in \cite[paragraph \textbf{4)} in \S{}1.6]{Bourdon} has given a very nice example showing sharpness of the previous statement: there is a group $\Gamma$ with $\ssl{\ell^pH}^1(\Gamma) \neq \{0\}$ for $p>2$ and an exact sequence $1 \to N \to \Gamma \to \zz \to 1$ where $N$ has $\IS_\omega$ and $\ssl{\ell^pH}^1(N) \neq \{0\}$ for $p>2$.

All of this (as well as \cite{moi-trans} and \cite{moi-lpharm}) seems to support a positive answer to \ref{laquestion}. Here is a probably easier question which should shed more light on this topic:
\begin{ques}
Is there an amenable group $G$ so that its Cayley graph has non-constant harmonic functions with gradient in $c_0$?
\end{ques}
A negative answer would give a positive answer to Question \ref{laquestion}. Note that this condition is much more restrictive than asking for harmonic functions of, say, sublinear growth. In fact, an answer to this question for $G$ solvable would already be interesting. Indeed, as pointed out by G.~Kozma, some wreath products have harmonic functions of sublinear growth.  

Another tempting path to a positive answer to Question \ref{laquestion}, comes from Corollary \ref{tvanspan-c}:
\begin{ques}
Given a Cayley graph of an amenable group of exponential growth, what is the largest $d \in \rr_{\geq 1}$ so that there is connected spanning subgraph which is Liouville and satisfies $\IS_d$?
\end{ques}
In fact, it would be sufficient to know what is the largest $d$ so that, for any pair of geodesic rays, there is connected subgraph [not necessarily spanning] containing these rays and which is Liouville and satisfies $\IS_d$.

Lastly, it seems very plausible that the inequalities $p < d/2$ could be changed to inequalities of the form $p<d$ by looking at more carefully defined transport plans. It is easy to see that this is true in some simple graphs.

{\it Acknowledgments:}
The author is grateful to M.~Bourdon, M.~de~la~Salle, P.~Pansu, T.~Pillon, M.~Puls, J.~C.~Sikorav, R.~Tessera and the anonymous referee for many useful comments and corrections to the previous versions. Warm thanks go to T.~Barthelmé and B.~de~Loynes for discussions about the Poisson and Martin boundaries.

\section{Definitions and further discussions}

\subsection{Preliminaries}\label{ss-prelim}

Recall that the conventions are that a graph $\Gamma = (\Gver,E)$ is defined by $\Gver$, its set of vertices, and $E$, its set of edges. All graphs will be assumed to be of bounded valency. The set of edges will be thought of as a subset of $\Gver \times \Gver$. The set of edges will be assumed symmetric (\ie $(x,y) \in E \imp (y,x) \in E$). Functions will take value in $\kk = \rr$ or $\mathbb{C}$. Functions on $E$ will often be anti-symmetric (\ie $f(x,y)= -f(y,x)$). This said $\ell^p(\Gver)$ is the Banach space of functions on the vertices which are $p$-summable, while $\ell^p(E)$ will be the subspace of functions on the edges which are $p$-summable. 

The gradient $\nabla:\kk^\Gver \to \kk^E$ is defined by $\nabla g(\gamma,\gamma') = g(\gamma') - g(\gamma)$. Given a finitely generated group $G$ and a finite set $S$, the Cayley graph $\cay(G,S)$ is the graph whose vertices are the element of $G$ and $(\gamma,\gamma') \in E$ if $\exists s \in S$ such that $s^{-1} \gamma = \gamma'$. This convention might be unusual from the point of view of random walks, but is much more convenient to write convolutions. In order for the resulting graph to have a symmetric edge set, $S$ is always going to be symmetric (\ie $s \in S \imp s^{-1} \in S$). Also, Cayley graphs are always going to be connected (\ie $S$ is generating). This said, it is worthwhile to observe that the gradient is made of $ \{(\lambda_s - \Id) g\}_{s \in S}$ where $\lambda$ is the left-regular representation. As for the right-regular representation, it is a (injective) homomorphism from $G$ into $\mathrm{Aut}\big(\cay(G,S) \big)$, the automorphism group of the Cayley graph.

The Banach space of $p$-Dirichlet functions is the space of functions $f$ on $\Gver$ such that $\nabla f \in \ell^p(E)$. It will be denoted $\D^p(\Gamma)$. In order to introduce the $\D^p(\Gamma)$-norm on $\kk^\Gver$, it is necessary to choose a vertex, denoted $e_\Gamma$ (in a Cayley graph, it is convenient to choose the neutral element). This said $\|f\|_{\D^p(\Gamma)}^p = \|\nabla f\|_{\ell^p(E)}^p + |f(e_\Gamma)|^p$. Lastly, $p'$ will denote the H\"older conjugate exponent of $p$, \ie $p' = p/(p-1)$ (with the usual convention that $1$ and $\infty$ are conjugate). 

The subject matter is the $\ell^p$-cohomology in degree one of the graph $\Gamma$. This is the quotient 
\[
\ell^p H^1(\Gamma) := ( \ell^p(E) \cap \nabla \kk^\Gver ) / \nabla \ell^p(\Gver). 
\]
This space is not always separated, and it is sometimes more convenient to look at the largest separated quotient, the reduced $\ell^p$-cohomology, 
\[
\ssl{\ell^p H}^1(\Gamma) := ( \ell^p(E) \cap \nabla \kk^\Gver ) / \srl{\nabla \ell^p(\Gver)}^{\ell^p(E)}. 
\]
By taking the primitive of these gradients, one may also prefer to define this by 
\[
\ssl{\ell^p H}^1(\Gamma):= \D^p(\Gamma) / \srl{ \ell^p(\Gver) + \kk}^{\D^p}. 
\]
A common abuse of language/notation will happen when we say the reduced cohomology is equal to the non-reduced one: this means that the ``natural'' quotient map $\ell^pH^1(\Gamma) \to \ssl{\ell^pH}^1(\Gamma)$ is injective.

When $G$ is a finitely generated group, this is isomorphic to the cohomology of the left-regular representation on $\ell^p(G)$, see Puls' paper \cite{Puls-harm} or Martin \& Valette \cite{MV}. Another important result is that $\ell^p$-cohomology is an invariant of quasi-isometry:
\begin{tea}\emph{(see Élek \cite[\S{}3]{El-qi} or Pansu \cite{Pan-qi})}
If two graphs of bounded valency $\Gamma$ and $\Gamma'$ are quasi-isometric, then they have the same $\ell^p$-cohomology (in all degrees, reduced or not).
\end{tea}
The result is actually much more powerful, in the sense that it holds for a large category of measure metric spaces (see above mentioned references). For shorter proofs in more specific situations see Puls \cite[Lemma 6.1]{Puls-pharmbnd} or Bourdon \& Pajot \cite[Théorème 1.1]{BP}. A first useful consequence is that it is possible to work on graphs and obtain results about manifold (or \emph{vice-versa}, when it is more convenient). A second corollary is that if $G$ is a finitely generated group, the $\ell^p$-cohomology of any two Cayley graphs are isomorphic. Consequently, one may speak of the $\ell^p$-cohomology of a group without making reference to a Cayley graph.

Kanai has shown \cite{Kanai} that any Riemannian manifold with Ricci curvature and injectivity radius bounded from below is quasi-isometric to a graph (of bounded valency). Also, if $M$ is a compact Riemannian manifold and $\widetilde{M}$ its universal covering, then $\widetilde{M}$ is quasi-isometric to the fundamental groups $\pi_1(M)$. Even if the language is that of graphs, there is always a corresponding result for Riemannian manifolds of bounded geometry (see Corollary \ref{tmanif-c} for a summary of the results expressed on Riemannian manifolds). 

Before moving on to the results, it is important to say a bit more about isoperimetric profiles (see \eg Woess' book \cite[(4.1) Definition]{Woe}). For a set of vertices $A$ let $\del A$ be the edges between $A$ and $A^\comp$. Let $\id{F}: \rr_{\geq 0} \to \rr_{\geq 0}$ be a function. Then $\Gamma$ (of bounded valency) satisfies the isoperimetric profile $\IS_\id{F}$ if there is a $\kappa>0$ such that, for any non-empty finite set of vertices $A$
\[
\id{F}(|A|) \leq \kappa |\del A|
\]
If $\id{F}(t) = t^{1-1/d}$ then this is called a $d$-dimensional isoperimetric profile ($d$ need not be an integer); the short notation is $\IS_d$. If $\id{F}(t) = t$ this is called a strong isoperimetric profile (or inequality); the short notation used here will be $\IS_\omega$. In the latter case, the constant $\kappa$ is sometimes referred to as the isoperimetric constant. 
It is straightforward to see that $\IS_\omega \imp \IS_d$ for all $d$. The converse is false (whence the notation with $\omega$ rather than $\infty$). Obviously $d' \leq d$, then $\IS_d \imp \IS_{d'}$.

Under the convention that Cayley graphs are always connected, recall that $G$ is a non-amenable group precisely when one (hence all) of its Cayley graph have a strong isoperimetric profile (F{\o}lner's classical result, \cite{Fol}). To see that a Cayley graph of a group satisfy $\IS_d$ for all $d$ if and only if the groups is not virtually nilpotent requires more effort (\eg one needs Gromov's theorem on groups of polynomial growth \cite{Gro-polyn}). See again \cite[\S{}14]{Woe} for more details.

The constant $\kappa$ in the various isoperimetric profiles is not an invariant of quasi-isometry. However (see again \cite[(4.7) Theorem]{Woe}), satisfying a $d$-dimensional or a strong isoperimetric profile is an invariant of quasi-isometry.

\subsection{Discussion}\label{ss-discu}

Theorem \ref{tpoisson} extends the result of Bourdon \& Pajot \cite[Théorème 1.1]{BP} from hyperbolic groups or spaces to those satisfying a $d$-dimensional isoperimetric profile, and extends the result of Lohoué \cite{Loh} from graphs with $\IS_\omega$ to graphs with $\IS_d$ for all $d$.

Tessera's \cite[Theorem 2.2]{Tes} showed that groups with CF also have vanishing of the reduced $\ell^p$-cohomology (in fact, for any weakly mixing representation, not just the left-regular representation). Let $B_n$ be the ball of radius $n$ around the identity element in some Cayley graph. An amenable group is CF if there exists a sequence $F_n \subset B_n$ of [finite] sets and a constant $K>0$ such that $\frac{|\del F_n|}{|F_n|} \leq K/n$.

To position Theorem \ref{tpoisson} with respect to Tessera's \cite[Theorem 2.2]{Tes}, note that many groups with CF were already known to have trivial Poisson boundary, see Kaimanovich's \cite[Theorem 3.3]{Kaimano} for $F \wr \zz$ and \cite[Corollary on p.23]{Kaimano} for polycyclic groups. Groups with CF have compression exponent $1$ (see Tessera \cite[Theorems 9 and 10]{Tes-cf}, so that using a bound of Austin, Naor \& Peres from \cite{ANP} on the speed exponent, one concludes that their Poisson boundary must be trivial. For some other result implying the Liouville property, see also Saloff-Coste \& Zheng \cite{SCZ} and \cite{moi-comp}.

However, $\zz_2 \wr \zz^2$ has a trivial Poisson boundary, but is not CF. This is a consequence of estimates of Erschler on the isoperimetric profile \cite{Ers} of wreath products. These estimates are not compatible with the isoperimetric profile of CF groups computed by Tessera in \cite{Tes-iso}. It has become frequent to show that a group is amenable by showing that it is Liouville. For example, Bartholdi \& Virág proved the Basilica group also has trivial Poisson boundary \cite[Theorem 1]{BV}. Since CF groups of exponential growth have return probability $\approx e^{Kn^{-1/3}}$ (see \cite{Tes-iso}), one may see that many Liouville groups are not CF.
For more examples, see the results of Revelle \cite{Rev}, Pittet \& Saloff-Coste \cite{PSCexp} and the references therein.

Let us also mention that groups of intermediate growth are most likely not CF (though the author ignores the existence of a direct argument). Any semi-direct product $N \rtimes H$ where $N$ is finitely generated nilpotent and $H$ is finitely generated and Liouville has trivial Poisson boundary by Kaimanovich \cite[Theorem 3.2]{Kaimano}.   

Theorem \ref{tpoisson} does not cover (for all $p$) groups of polynomial growth. But for these groups, many (quite different) proofs of the vanishing of reduced $\ell^p$-cohomology are available to the reader: groups of polynomial growth are quasi-isometric to nilpotent groups, these have infinitely many finite conjugacy class (in fact, infinite center) and see Kappos \cite[Theorem 6.4]{Kap} or \tcentrinf; they are also polycyclic, hence CF, and see Tessera \cite[Theorem 2.2]{Tes} or \tvantran; lastly they satisfy certain Poincaré inequalities and see Holopainen \& Soardi \cite[Corollary 1.10]{HS-polyn}. The first assertion requires to use that groups of polynomial growth are virtually nilpotent by Gromov's famous result \cite{Gro-polyn}.

Theorems  \ref{tinclu-t} and/or \ref{tpoisson} essentially unify many preceding notions of an ideal boundary which allows to compute the reduced $\ell^p$-cohomology (in degree one). These boundaries are $\ell^p$-corona (see Gromov \cite[\S{}8.C]{Gro} and Élek \cite{El-lpbnd}), the Bourdon \& Pajot boundary for hyperbolic spaces (see \cite{BP}), the Floyd boundary (Puls, see \cite{Puls-floyd}) and the $p$-harmonic boundary (Puls, see \cite{Puls-pharmbnd}). The advantage of the Poisson boundary is that it is better understood than most of the above (\eg it possesses a linear structure). It is worthwhile to underline that, in non-amenable groups, a result of Karlsson \cite{Karl-floyd} exhibits a strong link between Floyd and Poisson boundaries.

\tvantran~and \cite{moi-lpharm} shows many wreath products also have trivial reduced $\ell^p$-cohomology. This means that groups such as $H \wr \zz^k$ (where $H$, the ``lamp state'' group, is amenable and $k >0$) have no [bounded or not] harmonic functions with gradient in $\ell^p$ (though they have many bounded harmonic functions if $k>2$). Actually, the Poisson boundary of these groups is fully described by Erschler in \cite[Theorem 1]{Ers2} (under the further assumption that $k>4$), and one may directly check that these do not have harmonic functions in $\D^p$. 

It is also possible to show that certain semi-direct products (where $N$ is not finitely generated) have trivial reduced $\ell^p$ cohomology. More precisely, if $G = N \rtimes H$ is finitely generated, $H$ satisfies $\IS_d$, $H$ has trivial reduced $\ell^p$ cohomology and $N$ is not finitely generated as a group, then $G$ has also trivial reduced $\ell^p$ cohomology for $p<d/2$. See \cite{moi-lpharm} for more details. 

Also, using Erschler's result \cite[Theorem 2]{Ers2}, one may check that the free metabelian groups of rank $\geq 5$ also do not have harmonic functions with $\ell^p$ gradient, and hence trivial reduced $\ell^p$ cohomology in degree $1$. As a last note on this topic, Martin \& Valette \cite[Theorem.(iv)]{MV} shows that wreath products $H' \wr H$ where $H'$ is \emph{non-amenable} have trivial reduced $\ell^p$-cohomology. See \cite{moi-lpharm} for even more wreath products with trivial reduced $\ell^p$-cohomology in degree $1$. 

In higher degree, there is no hope to extend Theorem \ref{tpoisson} or \ref{tinclu-t}. Pansu computed in \cite[Théorème B]{Pan-rs} that, already in degree $2$, some groups have non-trivial cohomology exactly in an interval. 

Finally, let us sum up the results in the language of $p$-harmonic functions and on Riemannian manifolds. When $1<p<\infty$, it is known (see Puls \cite[\S{}3]{Puls-harm} or Martin \& Valette \cite[\S{}3]{MV}) that the existence of non-constant $p$-harmonic function (\ie $h \in \D^p(\Gamma)$ with $\nabla^* \maza_{p,p'} \nabla h = 0$, where $\maza_{p,p'}$ is the Mazur map defined by $(\maza_{p,p'} f)(\gamma) = |f(\gamma)|^{p-2} f(\gamma)$) is equivalent to the non-vanishing of reduced cohomology in degree $1$. In fact, up to a constant there is exactly one $p$-harmonic function in each reduced class. In this light, Theorem \ref{tpoisson} is even more surprising, as one may replace solutions of a non-linear equation (the $p$-Laplacian) by solutions to a linear one (the Laplacian).

Using the fact that the reduced $\ell^p$-cohomology of groups of polynomial growth is trivial, let's sum up the results for groups in terms of $p$-harmonic functions:
\begin{corn}\label{tpharm-c}
Let $p \in ]1,\infty[$ and $\Gamma$ be the Cayley graph of a finitely generated group $G$. Assume one of the following holds
\begin{itemize}\renewcommand{\labelitemi}{$\cdot$}\setlength{\itemsep}{0.2ex}
\item $G$ is amenable and $p \in ]1,2]$;
\item $\exists q \in [p,\infty]$ such that $\Gamma$ has no non-constant bounded harmonic function whose gradient is in $\ell^q(E)$ (\eg $\Gamma$ is Liouville)
\item $\exists q \in [p,d/2[$ such that there are no non-constant bounded $q$-harmonic functions on $\Gamma$.
\item there exists $N \lhd G$ such that $N$ is infinite, finitely generated, $G/N$ is infinite, $\ssl{\ell^pH}^1(N) =\{0\}$ and the growth of $N$ is at least a polynomial of degree $d>2p$;
\end{itemize}
Then there are no non-constant $p$-harmonic functions on $\Gamma$.
\end{corn}
Note that the corollary also holds 
\begin{itemize}\renewcommand{\labelitemi}{$\cdot$}\setlength{\itemsep}{0.2ex}
\item when $G$ has infinitely many finite conjugacy classes (\eg $G$ has infinite centre); 
\item when $G$ is transport amenable (\eg $G$ is of the form $H \wr \zz^k$ where $H$ is amenable and $k>0$);
\item when $G = N \rtimes H$ such that $G$ is finitely generated, $N$ is not finitely generated, $\ssl{\ell^pH}^1(H)=\{0\}$ and $H$ has $\IS_d$ for some $d >2p$;
\end{itemize}
thanks to \tvantran, \tcentrinf~ and \cite{moi-lpharm} (see also Holopainen \& Soardi \cite{HS-polyn}, Kappos \cite{Kap}, Martin \& Valette \cite{MV} and Tessera \cite{Tes}). 

Using quasi-isometry to re-express the results on Riemannian manifolds, one obtains:
\begin{corn}\label{tmanif-c}
Let $M$ be a Riemannian manifold (with Ricci curvature and injectivity radius bounded below). The degree one reduced $\ell^p$-cohomology of $M$ vanishes and (equivalently) there are no non-constant continuous $p$-harmonic functions on $M$ 
\begin{itemize}\renewcommand{\labelitemi}{$\cdot$}\setlength{\itemsep}{0.2ex}
\item if $M$ is the universal cover of a compact Riemannian manifold $M'$ with $\pi_1(M') =: G$ and $p$ satisfying one of the conditions of Corollary \ref{tpharm-c};
\item if $M$ satisfies a $d$-dimensional isoperimetric profile with $d >2p$ and $M$ is Liouville.
\item if $M$ satisfies a $d$-dimensional isoperimetric profile, $\ssl{\ell^qH}^1(M)=0$ for some $q \in [p, \infty[$ and $d >2q$.
\end{itemize}
Furthermore, $\ssl{\ell^1H}^1(M)=\{0\}$ if and only if $M$ has one end.
\end{corn}
In particular, if $M$ is the universal covering of $M'$ with $\pi_1(M')$ amenable, then $\ssl{\ell^pH}^1(M) = \{0\}$ for $p \in ]1,2]$.

Other known consequences of the triviality of the reduced $\ell^p$-cohomology include the triviality of the $p$-capacity between finite sets and $\infty$ (see Yamasaki \cite{Yam} and Puls \cite[Corollary 2.3]{Puls-pharmbnd}) and existence of continuous translation invariant linear functionals on $\D^p(\Gamma)/\kk$ (see \cite[\S{}8]{Puls-pharmbnd}). It also has implication on the possibility of realising the graph as a packing of spheres in $\rr^k$ (see Benjamini \& Schramm \cite{BS2}).

Since Lemma 4.4 from Holopainen \& Soardi \cite{HS} is quite important to the current methods and its proof relies on $p$-harmonic functions, the author feels he owes the reader a proof which does not require the use of $p$-harmonic functions. 
\setcounter{teo}{0}
\begin{lem}[Holopainen \& Soardi \cite{HS}, 1994]\label{tholoso-l}
Let $g \in \D^p(\Gamma)$ be such that $g \notin [0] \in \ssl{\ell^pH}^1(\Gamma)$. For $t \in \rr_{>0}$, let $g_t$ be defined as
\[
g_t(x) = \left\{ \begin{array}{ll}
g(x) & \text{if } |g(x)| < t, \\
t \tfrac{g(x)}{|g(x)|} & \text{if } |g(x)| \geq t.
\end{array} \right.
\]
Then there exists $t_0$ such that $g_t \notin [0]$, for any $t > t_0$. In particular, the reduced $\ell^p$ cohomology is trivial if and only if all bounded functions in $\D^p(\Gamma)$ have trivial classes.
\end{lem}
\begin{proof}
Assume, without loss of generality that $g(o) =0$ for some preferred vertex (\ie root) $o \in X$. Since $\|\nabla g\|_{\ell^\infty(E)} \leq  \|\nabla g\|_{\ell^p(E)} =: K$, given $x \in X$ and $P$ a path from $o$ to $x$, 
\[
|g(x)| = |g(x) - g(o)| = \sum_{e \in P:o \to x} \nabla g(e) \leq d(o,x) \|\nabla g\|_{\ell^p(E)}. 
\]
In particular, $g_t$ is identical to $g$ on $B_{t/K}$. Hence $\| g - g_t\|_{D^p(\Gamma)} \leq \|\nabla g\|_{\ell^p(B_{t/K}^\comp)}$, where $\ell^p(B_{t/K}^\comp)$ denotes the $\ell^p$-norm restricted to edges which are not inside $B_{t/K}$. Because $\nabla g \in \ell^p(E)$, $ \|\nabla g\|_{\ell^p(B_{t/K}^\comp)}$ tends to $0$, as $t$ tends to $\infty$.

Now if there is a infinite sequence $t_n$ such that  $g_{t_n}$ are in $[0]$ and $t_n \to \infty$, then $g_{t_n}$ is a sequence of functions in $[0]$ which tends (in $\D^p$-norm) to $g$. This implies $g \in [0]$, a contradiction. Hence, for some $t_0$, $g_t \notin [0]$ given that $t >t_0$. 
\end{proof}
It seems worthwhile to note that this proof works for $c_0$ (though it is clearly false for $\ell^\infty$, see Proposition \ref{tcohomli-p}), whereas the original proof requires $p$-harmonic functions (which are not defined in this extremal cases). Of course, one can check directly that $\D^1(\Gamma) \subset \ell^\infty(X)$ (see \S{}\ref{sapp}), so that even if the proof still works for $p=1$, this is not much of a surprise.

\section{Boundary values and simple random walks} \label{s-ranwal}

\setcounter{teo}{0}
\renewcommand{\theteo}{\thesubsection.\arabic{teo}}
\renewcommand{\thecor}{\thesubsection.\arabic{cor}}
\renewcommand{\theques}{\thesubsection.\arabic{ques}}

Let $\Gamma= (X,E)$ be a graph. Let $\xi_x \in \ell^1(X)$ be a family (as $x$ varies in $X$) of finitely supported, positive elements of $\ell^1$-norm one. Define
\[
\xi * g (x) = \int_X g(y) \dd \xi_x(y) = \sum_{y \in X} g(y) \xi_x(y).
\]
The main idea of this section will be to find a sequence $\xi^{(n)}_x$ of such families and suppose even that $\xi^{(n)} * g$ converges pointwise. It will be shown that elements of trivial class converge to a constant function. Using a good transport plan will allow to show the converse: if the limit is a constant function, the $g$ takes essentially one value at infinity. A classical truncation argument ensures such functions are of trivial class. 

Throughout the text, $[g]$ denote the class of $g$ in $\ssl{\ell^pH}^1(\Gamma)$. There should be no confusion as to which cohomology is considered.

Note that if there exists a $R \in \zz_{\geq 0}$ so that the support of $\xi_x$ is contained in a ball of radius $R$ at $x$, then $\xi * g \in [g]$. This seems to be decent motivation to look at such sequences. One could even be tempted to prove more than just pointwise convergence, namely, that there is convergence of $\xi^{(n)} * g$ in $\D^p(\Gamma)$. This second option is investigated in \cite{moi-trans}. It gives better results for small groups (\ie virtually nilpotent groups and some small wreath products), but is much more restrictive.

The first subsection treats a fairly generic situation and is more elaborated than strictly required. Existence of boundary values when $\xi^{(n)}_x$ is the time $n$ distribution of a simple random walker starting at $x$ can be significantly simplified. This can be done by looking at Remark \ref{rpans}.

\subsection{Boundary values}

Here is the avatar of this viewpoint: if $g$ does not take ``enough distinct values'' then its class is trivial. A function $h: X \to \kk$ will be said to tend to $c \in \kk$ as $|x| \to \infty$ if the following holds: there exists a $c \in \kk$ such that $\forall \eps >0$ the set $X \setminus h^{-1}(B_\eps(c))$ is finite (where $B_\eps(c) = \{ k \in \kk \mid |c-k|<\eps\}$). 

The following argument is well-known (it may probably be traced back to Strichartz \cite{Strich}, if not earlier) and motivates the introduction of boundary values. The proof is extremely similar to that of Lemma \ref{tholoso-l} (it also works for $c_0$ and $\ell^1$, but not $\ell^\infty$).
\begin{lem}\label{tbndvalintuit-l}
Assume $h \in \D^p(\Gamma)$ is such that $h \to c$ as $|x| \to \infty$. Then $h$ belongs to $\srl{\ell^p(X) + \kk}^{\D^p}$ (so $[h]=0$).
\end{lem}
\begin{proof}
First, one may assume $c=0$ by changing $h$ up to a constant. Then $\forall \eps >0$, the truncated function $h_\eps$ defined by 
\[
h_\eps(\gamma)= \left\{ \begin{array}{ll}
\eps h(\gamma) /|h(\gamma)| & \textrm{if }  |h(\gamma)| > \eps \\
h(\gamma) & \textrm{otherwise}
 \end{array} \right.
\]
is distinct from $h$ only on a finite set. Let $X_\eps = h^{-1}(B_\eps^\comp)$ be this finite set and $g_\eps = h-h_\eps$ be their difference (it is finitely supported, hence in $\ell^p(X)$ for any $p$). Note that $[g_\eps] = [0]$ and $\|h-g_\eps\|_{\D^1(\Gamma)} = \|h_\eps\|_{\D^1(\Gamma)}$. Furthermore,
\[
\nabla h_\eps \text{ is } \left\{ \begin{array}{ll}
\text{equal to } \nabla h & \text{ on } E \cap (X_\eps^\comp \times X_\eps^\comp), \\
\text{smaller in } |\cdot| \text{ than } \nabla h & \text{ on } \del X_\eps, \\
0  & \text{ on } E \cap (X_\eps \times X_\eps).
\end{array} \right.
\]
But $E \cap (X_\eps \times X_\eps)$ increases, as $\eps \to 0$, to the whole of $E$. The important consequence is that the $\ell^p$-norm of $\nabla h$ outside this set tends to $0$. Thus $\|h_\eps\|_{\D^p(\Gamma)} \to 0$ as $\eps \to 0$, and consequently $h \in \srl{\ell^p(X)}^{\D^p(\Gamma)}$.
\end{proof}
The aim here is to define a ``boundary value'' for functions so that the ``value'' does not depend on $p$ or the representative in the reduced cohomology class and it is constant exactly when the hypothesis of Lemma \ref{tbndvalintuit-l} apply. In order to do so, one must show some continuity in $\D^p$-norm. This will be done by an old trick in a new disguise: integration by parts under the cover of transportation problem.

For two finitely supported function $f$ and $g$ on a countable set $Y$, define the pairing $\langle f \mid g \rangle_Y = \sum_{y \in Y} f(y)g(y)$. The subscript $Y$ will often be dropped. This allows to define the adjoint of the gradient $\nabla$, denoted $\nabla^*$ and called divergence, by $\langle f \mid \nabla g \rangle_E = \langle \nabla^* f \mid g \rangle_X$. More precisely, for $f:E \to \kk$, one finds
\[
\nabla^*f(x) = \sum_{y \in N(x)} f(y,x) - \sum_{y \in N(x)} f(x,y)
\]
where $N(x)$ are the neighbours of $x$.
\begin{dfn}
A transport pattern from $\xi$ to $\phi$ (two finitely supported probability measures) is a finitely supported function on the edges $\tau_{\xi,\phi}$ such that $\nabla^* \tau_{\xi,\phi} = \xi - \phi$.
\end{dfn}
The name can be explained by the following simple fact: if $f$ represents any oriented path (with oriented edges counted with multiplicities if necessary) from a vertex $x$ to a vertex $y$ then $\nabla^* f = \delta_y - \delta_x$. As such, $\tau_{\xi,\phi}$ can be seen as the indication of how much mass is going to be transported through each edge in a transport from $\phi$ to $\xi$.
\begin{lem}\label{t-dphold-l}
Let $\xi$, $\phi$ be as above, $g \in \D^p(\Gamma)$ and $\tau_{\xi,\phi} \in \ell^{p'}(E)$ be a transport pattern. Then 
\[
\big| \smallint \! g \dd \xi  - \smallint \! g \dd \phi \big| \leq \|\nabla g\|_{\ell^p(E)} \|\tau_{\xi,\phi}\|_{\ell^{p'}(E)}.
\]
\end{lem}
\begin{proof}
Simply note that 
\[
\smallint g \dd \xi  - \smallint g \dd \phi 
  = \pgen{g \mid \xi - \phi }
  = \pgen{ g \mid \nabla^* \tau_{\xi,\phi} } 
  = \pgen{ \nabla g \mid \tau_{\xi,\phi} } ,
\]
then conclude using Hölder's inequality.
\end{proof}
As a consequence obtaining the following bound will become important:
\[\eqtag \label{eq-cond-rw}
\limm{n \to \infty} \supp{k \geq 0} \|\tau_{\xi^{(n)}_x,\xi^{(n+k)}_x}\|_{\ell^{r}(E)} = 0.
\]
where $\xi^{(n)}_x$ is a sequence of (finitely supported probability) measures. 
\begin{dfn}
Let $\{\xi^{(n)}_x\}_{x \in X, n \in \zz_{\geq 0}}$ be (finitely supported probability) measures such that \eqref{eq-cond-rw} holds for $r=p'$ and for all $x \in X$. If $g \in \D^p(\Gamma)$, then $\xi^{(n)} * g$ converges pointwise and the resulting pointwise limit is called the boundary value of $g$ for $\xi^{(n)}$.
\end{dfn}
The above definition of boundary value is closer to that of Pansu \cite{Pan} than Bourdon \& Pajot \cite{BP}. With the choice of $\xi^{(n)}_x$ which will considered in a few paragraphs, these boundary values are harmonic functions. 

An important point in the preceding definition is that, if the condition \eqref{eq-cond-rw} holds for $r=p'$ it holds for $r=q'>p'$, and if $g \in \D^q(\Gamma)$ then $g \in \D^p(\Gamma)$. In other words, if this boundary value exists for $p$, it does not depend on $p$ and exists for all $q<p$.

The choice of $\tau$ is very important in the computation of the norm. Even in $\zz^2$ and for $\xi^{(n)}_x$ the normalised characteristic function of the balls around $x$, some (too simple) choices of $\tau$ will not satisfy the required condition whereas others will.

\subsection{Simple random walks}\label{ss-SRW}

In order to start with the easiest setting, the measure $\xi^{(n)}_x$ will, in this subsection and subsection \ref{ss-incvan}, be $P^{(n)}_x$ where $P^{(n)}_x(y)$ is the probability a simple random walker starting at $x$ reaches $y$ in $n$ steps. With this vocabulary, $P^{(n)} * g(x)$ is the expected value of $g$ after $n$ steps of a simple random walk starting at $x$. When the boundary value exists, it is a harmonic functions (for the simple random walk); see also Remark \ref{rpans} below.

The advantage of these function is that there is a natural (albeit very unefficient) choice for $\tau_{P^{(n)}_x, P^{(n+k)}_x}$. Namely, $\tau_{P^{(n)}_x, P^{(n+k)}_x}$ is given by continuing the random walk $k$ steps. Let us dwell on an explicit realisation of this idea.

First, if $\xi, \phi$ and $\psi$ are finitely supported probability measures, $\tau_{\xi,\phi}$ is a transport pattern from $\phi$ to $\xi$ and $\tau_{\psi,\xi}$ is a transport pattern from $\xi$ to $\psi$, then $\tau_{\psi,\phi} = \tau_{\xi,\phi} + \tau_{\psi,\xi}$. This follows trivially from the linearity of $\nabla^*$. For our current purpose, it will be sufficient to define the transport plan $\tau_{P^{(n)}_x, P^{(n+1)}_x}$, because one can then pick $\tau_{P^{(n)}_x, P^{(n+k)}_x} =  \sum_{i=n}^{n+k-1} \tau_{P^{(n)}_x, P^{(n+1)}_x}$.

Now let $P^{(i)}_{x,E}$ be the measure on the edges defined by
\[
\forall y \in X, \quad \forall z \in N(y), \qquad  P^{(i)}_{x,E}(y,z) = \tfrac{1}{|N(y)|} P^{(i)}_x(y).
\]
The claim is that this is a candidate for $\tau_{P^{(i)}_x, P^{(i+1)}_x}$. To verify this, fix a $y \in X$ and look at
\[
\sum_{z \in N(y)} P^{(i)}_{x,E}(y,z) = P^{(i)}_{x}(y) \quad \text{ and } \quad \sum_{z \in N(y)} P^{(i)}_{x,E}(z,y) = \sum_{z \in N(y)} \tfrac{1}{|N(z)|} P^{(i)}_x(z) = P^{(i+1)}_{x}(y).
\]
The first equality follows from the definition of $P^{(i)}_{x,E}$, the second from the definition of a simple random walk. The difference of these sum is the value of $\nabla^*P^{(i)}_{x,E}$ and shows the claim.

This formula is particularly useful as one can give an upper bound in terms of more well-studied quantities:
\[
\| \tau_{P^{(n)}_x, P^{(n+k)}_x} \|_{\ell^{p'}(E)} \leq \sum_{i=n}^{n+k-1} \|P^{(i)}_x\|_{\ell^{p'}(X)}.
\]
When all vertices have the same valency, one can take a slightly smaller constant in front of the sum on the right-hand side, but it is completely irrelevant for the present purpose. This proves:
\begin{lem}\label{t-bornnormrw-l}
If $\xi^{(n)}_x = P^{(n)}_x$ as above, and, for all $x\in X$,  $\sum_{i=0}^{\infty} \|P^{(i)}_x\|_{\ell^{p'}(X)} < +\infty$, then any $g \in \D^p(\Gamma)$ admits a boundary value for $\xi^{(n)}$.
\end{lem}
In particular, the simple random walk is transient exactly when the condition of Lemma \ref{t-bornnormrw-l} holds for $p' = \infty$. (It never holds in the case $p'=1$, but $p=\infty$ is also not of interest.)

Fortunately, there are very good estimates at hand for $\|P^{(i)}_x\|_{\ell^{p'}(X)}$, and some of them rely only on isoperimetric profiles (and in the case of a Cayley graph, the growth of the group). 
Indeed, if $\Gamma$ has $\IS_d$, then
\[
\exists K>0, \quad \forall x,y \in X, \quad  P^{(n)}_x(y) \leq K n^{-d/2},
\]
see Woess' book \cite[(14.5) Corollary]{Woe}.
\begin{cor}
If $\Gamma$ has $\IS_d$, then boundary values of $g \in \D^p(\Gamma)$ exists for $\xi^{(n)}_x = P^{(n)}_x$ if $p < d/2$.
\end{cor}
\begin{proof}
Obviously $\| P^{(n)}_x \|_{\ell^1(X)}=1$. By H{\"o}lder's inequality, 
$$\| P^{(n)}_x \|_{\ell^{p'}(X)}  \leq \| P^{(n)}_x \|_{\ell^1(X)}^{1/p'}  \| P^{(n)}_x \|_{\ell^\infty(X)}^{1/p}.$$
Hence $\| P^{(n)}_x \|_{\ell^{p'}(X)} \leq K' n^{-d/2p}$ uniformly in $x$, for some $K'>0$.
\end{proof}
This shows that there are plenty of graphs where boundary values may be defined. It remains to prove these boundary values have the desired properties.
\begin{lem}\label{ttrivbnd-l}
Assume $\Gamma$ has $\IS_d$ and $p < \tfrac{d}{2}$. If $g \in \srl{\ell^p(X)}^{\D^p(\Gamma)}$ then its boundary value is trivial. 
\end{lem}
\begin{proof}
If $g \in \ell^p(\Grou)$ then $P^{(n)}_x g$ tends to $0$ (as $\ell^p \subset c_0$ and the mass of $P^{(n)}$ tends to $0$ on finite sets). It remains to be checked that convergence in $\D^p(\Gamma)$ does not alter the boundary value. However, by Lemma \ref{t-dphold-l}, if $g-g_n$ tends to $0$ in $\D^p(\Gamma)$ norm so does the difference of their boundary values.
\end{proof}
In fact, since the boundary value of a sum is the sum of boundary values, this shows the boundary value does not depend on the representative of the reduced cohomology class, up to a constant function. 
\begin{lem}\label{tbndtriv-l}
Assume $g\in \D^p(\Gamma)$, $\Gamma$ has $\IS_d$ and $p < \tfrac{d}{2}$. If the boundary value of $g$ for $P^{(n)}$ is a constant function, then $g \to c$ as $|x| \to \infty$ and (consequently) $[g]=0 \in \ssl{\ell^pH}^1(\Gamma)$.
\end{lem}
\begin{proof}
Up to changing $g$ by a constant (this does not affect its cohomology class), one may assume that the limit takes value $0$ everywhere. Fix a root $o \in X$ of the graph. For any $\eps>0$, let $n_\eps$ be such that $\|\nabla g \|_{\ell^p(E \setminus B_{n_\eps}(o))} < \eps$ and, uniformly in $x$, $\sum_{i \geq n_\eps} \|P^{(i)}_x\|_{\ell^{p'}(E)} < \eps$. 

Let $E[Y]$ denote the edges incident with $Y \subset X$ and $\pgen{. \mid .}_{E'}$ restriction of the pairing to the set $E'$. If $x \notin B_{3n_\eps}(o)$, then
\[
\begin{array}{rl}
|P^{(k)} * g(x) - g (x) | 
  &= |\pgen{ g \mid P^{(k)}_x - \delta_x }| \\
  &= |\pgen{ g \mid \nabla^* \tau_{P^{(k)}_x, \delta_x } }|\\
  &= |\pgen{ \nabla g \mid \tau_{P^{(k)}_x, \delta_x } }|\\
  & \leq | \pgen{ \nabla g \mid \tau_{P^{(k)}_x, \delta_x } }_{E[B_{n_\eps}(x)]}| +|\pgen{ \nabla g \mid \tau_{P^{(k)}_x, \delta_x } }_{E[B_{n_\eps}]^\comp} | \\
  & \leq \| \nabla g\|_{\ell^p(E[B_{n_\eps}(x)])} \|\tau_{P^{(k)}_x, \delta_x} \|_{\ell^{p'}(E)} + \| \nabla g\|_{\ell^p(E )} \|\tau_{P^{(k)}_x, \delta_x}\|_{\ell^{p'}(E \setminus B_{n_\eps}(x))} \\
  & \leq \eps  \sum_{i \geq 0} \|P^{(i)}_x\|_{\ell^{p'}(E)}  + \| \nabla g\|_{\ell^p(E)}  \sum_{i \geq n_\eps} \|P^{(i)}_x\|_{\ell^{p'}(E)} \\
  & \leq \eps  \sum_{i \geq 0} \|P^{(i)}_x\|_{\ell^{p'}(E)}  + \| \nabla g\|_{\ell^p(E)} \eps \\
  & \leq c \eps ,
\end{array} 
\]
where $c$ is a constant depending only on the constant in $\IS_d$ and the $\D^p$-norm of $g$. Thus, letting $k \to \infty$, for all $x \notin B_{3n_\eps}(e)$, $|g(x)| \leq c \eps$. Thus, Lemma \ref{tbndvalintuit-l} may be applied to yield that $[g]=0$.
\end{proof}
A trivial, but useful, remark, is that if $g$ is bounded (\ie in $\ell^\infty(X)$) then its boundary value is also bounded.

\subsection{Inclusion and vanishing}\label{ss-incvan}

\begin{cor}\label{tpoiss1-c}
Let $\Gamma$ be a graph with $\IS_d$ and $1\leq q \leq p < d/2$. Then the natural quotient $\ssl{\ell^qH}^1(\Gamma) \to \ssl{\ell^pH}^1(\Gamma)$ is an injection. 
\end{cor}
\begin{proof}
Assume $[g] \neq [0] \in \ssl{\ell^qH}^1(\Gamma)$. Then, by Lemma \ref{tbndtriv-l}, its boundary value is not trivial. However, this boundary value does not depend on $p$ and so, by Lemma \ref{ttrivbnd-l}, $g$ is not trivial in $\ssl{\ell^pH}^1(\Gamma)$.
\end{proof}
Finally, let us consider the case where the graph has is Liouville, \ie there are no non-constant bounded harmonic functions. This means that, if $g \in \ell^\infty(X)$ and $P^{(n)}* g$ converges pointwise, then the limit is a constant function.
\begin{cor}\label{tpoiss2-c}
Let $\Gamma$ be a Liouville graph with $\IS_d$ and $1\leq q \leq p < d/2$. , then $\ssl{\ell^pH}^1(\Gamma) = \{0\}$ for all $p \in [1,\infty[$.
\end{cor}
\begin{proof}
By the truncation lemma from Holopainen \& Soardi (see Lemma \ref{tholoso-l}), it suffices to show all bounded functions in $\D^p(\Gamma)$ are trivial. But if $g$ is in $\ell^\infty(X)$ and the graph is Liouville, the boundary value of $g$ for $P^{(n)}_x$ is constant. By Lemma \ref{tbndtriv-l}, the conclusion follows.
\end{proof}

\begin{rmk}\label{rpans}
As P.~Pansu pointed out to the author, a result of Lohoué \cite{Loh} shows that in graphs satisfying $\IS_\omega$, there is a harmonic element in each (unreduced) class. This element is the boundary value above. To see this, first recall that the representative in the (reduced or unreduced) class of $g$ exhibited by N.~Lohoué is defined by $g+u$ where $u= \Delta^{-1} (-\Delta g)$, where $\Delta = \Id - R$ and $R$ is the random walk (or averaging operator): $Rg = P^{(1)}*g$. Thanks to H.~Kesten, graphs satisfying $\IS_\omega$ are exactly those where  $\|R\|_{\ell^2 \to \ell^2} < 1$. On the other hand, $\|R\|_{\ell^p \to \ell^p} \leq 1$ for $p=1$ or $\infty$. Thus $\|R\|_{\ell^p \to \ell^p} < 1$ for any $p\in ]1,\infty[$ by Riesz-Thorin interpolation, and $\Delta^{-1} = \sum_{i \geq 0} R^i$ is bounded from $\ell^p(X)$ to itself. Next notice that
\[
\begin{array}{rlll}
\tilde{g} - g 
  &= \limm{n \to \infty} P^{(n)}* g - g 
  &= \limm{n \to \infty} R^n g - g \\
  &= \sum_{i \geq 0} (R^{i+1} g - R^i g) 
  &= \sum_{i \geq 0} R^i (R - \Id) g  \\
  &=  ( \sum_{i \geq 0} R^i ) * (-\Delta g),
\end{array}
\]
to conclude that $\tilde{g} = g+u$.
\end{rmk}

This argument may not apply in (many, if not all) Cayley graphs of amenable groups. Indeed, this would mean the harmonic function belongs to the same (unreduced!) cohomology class. But, in amenable groups, the reduced and unreduced $\ell^p$-cohomologies are never equal (see Guichardet \cite[Corollaire 1]{Guich}). Thus, if the reduced $\ell^p$-cohomology is trivial (which is already known for many amenable groups), this would give a contradiction. 

However, it is still possible to answer positively a weaker form of P.~Pansu's question \cite[Question 6 in \S{}1.9]{Pan-rs} (\ie is the absence of non-constant harmonic function whose gradient has finite $\ell^p$-norm is equivalent to $\ssl{\ell^pH}^1(\Gamma)=\{0\}$). Lemma \ref{tbndtriv-l} shows that if there is such a harmonic function and $\Gamma$ has $\IS_d$ for $d >2p$, then $\ssl{\ell^pH}^1(\Gamma)$ is not trivial. Indeed, such a function would be its own boundary value, and being non-constant, it would be non-trivial in cohomology.

Let us now address the reverse implication.

\begin{lem}\label{tquespan-l}
Let $\Gamma$ be a graph with $\IS_d$ and $1\leq p < d/2$. For any $g \in \D^p(\Gamma)$, let $\tilde{g}$ be its boundary value for $P^{(n)}_x$. Then $\tilde{g}$ is in the same $\ell^qH^1(\Gamma)$ class as $g$ for all $q> \tfrac{dp}{d-2p}$.  
\end{lem}
\begin{proof}
As before, write:
\[
\tilde{g} - g = \limm{n \to \infty} P^{(n)} *g - g = \sum_{i \geq 0} (P^{(i+1)} *g - P^{(i)}* g) = \sum_{i \geq 0} P^{(i)} *(P^{(1)} - \Id) * g 
\]
Let $h = (P^{(1)}- \Id)* g = \Delta g \in \ell^p(X)$ and remember $P^{(i)}$ are operators defined by a kernel. Using Young's inequality (see \eg Sogge's book \cite[Theorem 0.3.1]{Sogge}) for $r>p$ and $1+\tfrac{1}{r} = \tfrac{1}{p} + \tfrac{1}{q}$,
\[
\begin{array}{rll}
\| \tilde{g} - g \|_{\ell^r(X)} 
  &= \| (\sum_{i \geq 0} P^{(i)}) * h  \|_{\ell^r(X)} 
  &\leq \supp{x \in X} \| \sum_{i \geq 0} P^{(i)}_x \|_{\ell^q(X)} \| h  \|_{\ell^p(X)} \\
  &\leq 2D  \supp{x \in X} \| \sum_{i \geq 0} P^{(i)}_x \|_{\ell^q(X)} \|\nabla g\|_{\ell^p(E)}
\end{array}
\]
where $D$ is the maximal valency of a vertex. Since $\supp{x \in X} \| \sum_{i \geq 0} P^{(i)} \|_{\ell^q(X)} < +\infty$ for all $q'<d/2$, this means $\tilde{g}- g \in \ell^r(X)$ (for all $r> \tfrac{dp}{d-2p}$) and consequently that $\tilde{g}$ and $g$ belong in the same (unreduced) $\ell^r$-cohomology class.
\end{proof}

As mentioned before, to address the question of whether $\tilde{g}$ and $g$ are in the same reduced $\ell^p$-cohomology class, the author believes one would need to consider question similar to those of the transport problem from \cite{moi-trans}.

\begin{cor}\label{tpoiss3-c}
Let $\Gamma$ be a graph with $\IS_d$ and $1\leq p < d/2$. If there are no non-constant bounded harmonic functions whose gradient has finite $\ell^p(E)$ norm, then $\ssl{\ell^qH}(\Gamma)= \{0\}$ for all $q < \tfrac{pd}{d+2p}$. 

Conversely, if there is a  non-constant harmonic functions whose gradient has finite $\ell^p(E)$ norm, then $\ssl{\ell^pH}(\Gamma)\neq \{0\}$ and there is also a \emph{bounded} non-constant harmonic functions whose gradient has finite $\ell^r(E)$ norm for $r > \tfrac{pd}{d-2p}$. 
\end{cor}
\begin{proof}
Under this new hypothesis, one has that $\tilde{g}$ is constant (by Lemma \ref{tquespan-l}), so Lemma \ref{tbndtriv-l} implies that $[g]=0 \in \ssl{\ell^pH}^1(\Gamma)$. 

As mentioned above, the converse is a consequence of Lemma \ref{tbndtriv-l}. Boundedness may be added thanks to the Holopainen \& Soardi lemma (see Lemma \ref{tholoso-l}), but upon taking boundary values again the gradient might loose regularity and be only in $\ell^r$ for some bigger $r$.
\end{proof}
Corollaries \ref{tpoiss1-c}, \ref{tpoiss2-c} and \ref{tpoiss3-c} are particularly interesting for graphs satisfying $\IS_d$ for every $d \in \zz_{\geq 1}$. This will be developed in subsection \ref{ss-varia}.

\section{Consequences}\label{s-caygra}

This section finishes the proof of theorems announced in the introduction, namely Theorems \ref{tinclu-t} and \ref{tvanquot-t}.

\subsection{Graphs having $\IS_\omega$}\label{ss-graiso}

\begin{proof}[Proof of Theorem \ref{tinclu-t}]
When the graph has $\IS_\omega$, the fact that the quotient is always separated (\ie $\ssl{\ell^q H}^1(\Gamma) = \ell^q H^1(\Gamma)$) is a consequence Remark \ref{rpans}.

Let $\Gamma =: \Gamma_1$. For any finite subset $F$ of $\Gamma_1$, $| \del^{\Gamma_1} F | \geq c |F|$ for some $c>0$. So take $n \in \nn$ such that $nc \geq 1$. Let $\Gamma_k$ be graph obtained by adding, for all $\gamma \in \Gamma$, an edge between $\gamma$ and all vertices at distance $\leq n$ from $\gamma$ (\ie the $k$-fuzz of $\Gamma$). Denote $\del^{\Gamma_k}F$ to be the boundary of $F$ in $\Gamma_k$ and $B_k(F)$ to be the sets of vertices at distance $\leq k$ from $F$ ($F$ included). Then, for any finite subset $F$ of $\Gamma_n$, 
\[
|\del^{\Gamma_n} F| \geq |\del^{\Gamma_1} B_n(F)| + |\del^{\Gamma_{n-1}}F| \geq c |F| + |\del^{\Gamma_{n-1}}F| \geq \ldots \geq nc |F| \geq |F|.
\]
In other words, there is a graph $\Gamma_n$, quasi-isometric to $\Gamma_1$, satisfying a strong isoperimetric profile with constant $\geq 1$. But by a result of Benjamini and Schramm \cite[Theorem 1.2]{BS}, there is then a spanning tree in $\Gamma_n$ which is obtained by adding edges between disjoint copies of binary trees. Let $T$ be this tree, and look at the simple random walk on this tree. The boundary value of $g$ for this tree will be in the same class for $\ell^pH^1(\Gamma)$ as $g$ (the difference being in $\ell^p(X)$, the class is preserved for any set of edges as long as the valency remains bounded). 

Furthermore, the fact that the boundary value is constant depends only on the boundary value for the tree. Boundary values actually shows the existence of a commuting diagram of ``natural'' injections (where $T$ is the spanning tree inside $\Gamma_n$):
\[
\begin{array}{ccc}
\ell^qH^1(\Gamma_n) & \to  & \ell^qH^1(T) \\
   \downarrow     &      & \downarrow \\
\ell^pH^1(\Gamma_n) & \to  & \ell^pH^1(T) \\
\end{array}
\]
As a consequence, the only boundary values possible (in a graph $\Gamma$ having  $\IS_\omega$) are those of $T$. To know if a boundary value of the tree is an actual cohomology class in $\Gamma$, it remains to check that it belongs to $\D^p(\Gamma)$.
\end{proof}

\subsection{Further corollaries}\label{ss-varia}

One of the most important consequence of boundary values is that they really give an idea of how the function behaves at infinity. It is already nice that it does not depend on $p$ (up to some value), but it in fact also does not depend so much on the graph. Recall that a spanning subgraph $H$ of $\Gamma$ is a graph on the same vertices but with some edges removed. Two things are easy to check. First, if a spanning subgraph $H$ has $\IS_d$, then so does the full graph $\Gamma$. Second, if $g \in \D^p(\Gamma)$ then $g \in \D^p(H)$.

However, $\ssl{\ell^pH}^1(H) = \{0\}$ has, in general, no incidence on the triviality of $\ssl{\ell^pH}^1(\Gamma)$. For example, there are many graphs with a spanning line (or half-line) which have non-trivial $\ell^p$-cohomology.
\begin{cor}\label{tvanspan-c}
Let $H$ be a connected spanning subgraph of $\Gamma$ which has $\IS_d$ (hence $\Gamma$ also has $\IS_d$) and assume $p < d/2$. Then if $g \in \D^p(\Gamma)$ (hence in $\D^p(H)$) is such that $[g] = 0 \in \ssl{\ell^pH}^1(H)$ then $[g] = 0 \in \ssl{\ell^pH}^1(\Gamma)$.

In particular, if $\ssl{\ell^pH}^1(\Gamma) \neq \{0\}$ then $\ssl{\ell^pH}^1(H) \neq \{0\}$.
\end{cor}
\begin{proof}
Indeed, if $[g] = 0 \in \ssl{\ell^pH}^1(H)$ then, by Lemmas \ref{ttrivbnd-l} and \ref{tbndtriv-l}, then $g \to c$ as $|x| \to \infty$. This statement remains true in $\Gamma$ (because $H$ is spanning and connected), hence, by Lemma \ref{tbndvalintuit-l}, the conclusion follows.
\end{proof}
The main application of this idea is done in \S{}\ref{ssarggab}, \ie the proof of Theorem \ref{tvanquot-t}.
\begin{rmk}\label{runiq}
Recall that there is a unique (up to a constant) $p$-harmonic function in each reduced $\ell^p$-cohomology class. The existence of a harmonic function in reduced $\ell^p$-cohomology class seems unclear (at least to the author). However, it is easy to see that, if it exists and $p < \tfrac{d}{2}$, it is unique (up to constants). Indeed, assume $h_1$ and $h_2$ are two harmonic functions in $\D^p(\Gamma)$ and $\Gamma$ has $\IS_d$ for $d> 2p$. Then $h_1-h_2$ is harmonic and of trivial class. By Lemma \ref{tbndtriv-l}, this means that it is constant at infinity. But a harmonic function which is constant at infinity is also constant by the maximum principle. Hence, $h_1-h_2$ is a constant function.
\end{rmk}
The upcoming corollary is just to make a cleaner statement in the case of groups of superpolynomial growth (the case of polynomial growth being well-understood, see \S{}\ref{s-intro} or \S{}\ref{ss-discu}). N.~Varopoulos showed that superpolynomial growth of a group implies superpolynomial decay of $\|P^{(i)}\|_{\ell^\infty(\Grou)}$ (\eg see \cite{Var} or Woess' book \cite[(14.5) Corollary, p.148]{Woe}. 
In particular, Cayley graphs of groups of polynomial growth of degree $d$ satisfy $\IS_d$ and Cayley graphs of groups of superpolynomial growth will satisfy $\IS_d$ for any $d \in \rr_{\geq 1}$. In the latter case, Corollaries \ref{tpoiss1-c}, \ref{tpoiss2-c} and \ref{tpoiss3-c} yield
\begin{cor}
Let $G$ be a group of superpolynomial growth and $\Gamma$ a Cayley graph. Let $1\leq p <\infty$. Then there exists a map $\pi$ from $\D^p(\Gamma)$ to the space of harmonic functions modulo constants such that
\begin{enumerate}
\item $\pi(g) = \pi(h) \iff [g]=[h] \in \ssl{\ell^pH}^1(\Gamma)$;
\item $g \in \ell^\infty(G) \implies \pi(g) \in \ell^\infty$;
\item $\pi(g) \in \D^q(\Gamma)$ for any $q>p$.
\end{enumerate}
\end{cor}
Using Theorem \ref{tpoisson}, it also easy to find \emph{graphs} of polynomial growth which have $\ssl{\ell^pH}^1(\Gamma) \neq \{0\}$.
\begin{exa}
Take $\Gamma$ to be two copies of the usual Cayley graph of $\zz^d$ (where $d \geq 3$) and join them by an edge (say between their respective identity elements). So a vertex $x \in \Gamma$ can be written as $x = (z,i)$ with $i \in {1,2}$ and $z \in \zz^d$. Consider $f = \sum_{i \geq 0} P^{(i)}_e$ on $\zz^d$ (where $e \in \zz^d$ is the identity element). Define 
\[
g(z,i) = \left\{\begin{array}{ll}
f(z) & \text{if } i = 1,\\
K+2f(0) - f(z) & \text{if } i=2,
\end{array}\right.
\]
where $K = \nabla^* \nabla f(0)$. Then $g$ is harmonic and $\nabla g \in \ell^p(E)$ for $p<2d$ (since $f \in \ell^p(\zz^d)$ for $p<2d$). Since $g$ is non-constant, $\ssl{\ell^pH}^1(\Gamma) \neq \{0\}$ for any $p<d/2$. 

Actually, since $\Gamma$ has two ends, using Proposition \ref{tcohoml1-p} one has that $\ssl{\ell^1H}^1(\Gamma) \neq \{0\}$. Since a corollary of Theorem \ref{tpoisson} is that $\ssl{\ell^qH}^1(\Gamma) \inj \ssl{\ell^pH}^1(\Gamma)$ for $1 \leq q \leq  p< d/2$, one also sees that $\ell^pH^1(\Gamma) \neq \{0\}$ for $p \in [1,d/2[$.
\end{exa}
In fact, the previous example could have been done for any Cayley graph (as soon as the growth is faster than quadratic).

Another quotient which is sometimes studied is the $\ell^{p,q}$-cohomology. This is the quotient,
\[
\ell^{p,q} H^1(\Gamma):= \dfrac{\D^p(\Gamma) }{ \ell^q(\Gver) + \kk}. 
\]
Recall that in Lemma \ref{tquespan-l}, if $g \in \D^p(\Gamma)$ and $\Gamma$ has $\IS_d$, then $\tilde{g} - g \in \ell^q(X)$ for $q > \tfrac{dp}{d-2p}$. As a corollary
\begin{cor}
If $\Gamma$ is Liouville and has $\IS_d$, then $\ell^{p,q}H^1(\Gamma)$ is trivial for all $q > \tfrac{dp}{d-2p}$.
\end{cor}
There is also an analogue of Theorem \ref{tinclu-t}:
\begin{prop}
Assume $\Gamma'$ has $\IS_d$ and $\Gamma'$ is quasi-isometric to a spanning subgraph of $\Gamma$. Let $p< \infty$ and $q > \tfrac{dp}{d-2p}$. Then non-trivial $\ell^{p,q}$-classes of $\Gamma$ are given by functions $g \in \D^p(\Gamma)$ whose class is non-trivial in $\ell^{p,q}H^1(\Gamma')$.
\end{prop}

\subsection{Normal subgroups}\label{ssarggab}

The aim here is to prove Theorem \ref{tvanquot-t}. A preliminary result on graphs ``stitched'' together is necessary.
\begin{lem}\label{tgamn-l}
Assume $d/2 > p \geq 1$. Let $\Gamma_i$ be a family of graphs all satisfying $\IS_d$ with the same constant. Let $\Gamma^\Pi$ be the disjoint union of the $\Gamma_i$. Fix some $k \in \zz_{\geq 2}$ and let $\Gamma'$ be obtained by adding $\leq k$ edges to each $\Gamma_i$ inside $\Gamma^\Pi$ so that the resulting graph is connected. Then $[g] \neq 0 \in \ssl{\ell^pH}^1(\Gamma')$ if and only at least one of the following holds:
\begin{itemize}
\item there exists $i$ such that $[g_{\mid \Gamma_i}] \neq 0 \in \ssl{\ell^pH}^1(\Gamma_i)$;
\item there exists $i \neq j$ such that $g_{\mid \Gamma_i}$ and $g_{\mid \Gamma_j}$ are constant at $\infty$ (in the sense of Lemma \ref{tbndvalintuit-l}) but not the same constant.
\end{itemize}
If the $\Gamma_i$ have $\IS_\omega$, the statement holds for unreduced cohomology.
\end{lem}
\begin{proof}
If all $\Gamma_i$ satisfy $\IS_d$ with the same constant then so do $\Gamma^\Pi$ and $\Gamma'$ (a consequence of the concavity of $t \mapsto t^{1-1/d}$). 

Without loss of generality, one may assume that $g$ is bounded (by the Holopainen \& Soardi truncation lemma, see Lemma \ref{tholoso-l}). If $g$ is non-trivial in cohomology, then, by Lemma \ref{tbndtriv-l}, the boundary value of $g$ in $\Gamma'$ is not constant. In particular, $g \not \to c$ as $|x| \to \infty$. 

On the other hand, one could see $g$ and $\widetilde{g}$ (its boundary value in $\Gamma'$) as functions on $\Gamma^\Pi$ or restrict them to the $\Gamma_i$. In each $\Gamma_i$, note that there are at most $k$ vertices which are connected to some other $\Gamma_j$ in $\Gamma'$. Since the $\Gamma_i$ are transient ($d >2$), the probability that a random walker starting at $x$ hits one of those $k$ vertices decreases to $0$ as $x$ tends to infinity. This means that the boundary value of $g_{|\Gamma_i}$ will (for large enough $x$) be arbitrarily close to $(\widetilde{g})_{|\Gamma_i}$.

Thus, a first possibility is that the boundary of $g$ restricted to one of the $\Gamma_i$ will be non-trivial. If this is not the case, then $g_{|\Gamma_i} \to c_i$ as $|x| \to \infty$. Since the boundary value of $g$ in $\Gamma'$ is non-trivial, the constants $c_i$ must be different. 

Similarly, if all $[g_{\mid \Gamma_i}]=0$, then, by Lemma \ref{ttrivbnd-l}, all boundary values of the copies of $\Gamma_i$ are constant. By Lemma \ref{tbndtriv-l}, $g$ is trivial exactly when these constant are the same.

Lastly, if the $\Gamma_i$ have $\IS_\omega$, boundary values of $g$ are in the same unreduced class as $g$ (see Remark \ref{rpans}), so unreduced cohomology may be considered.
\end{proof}

\begin{proof}[Proof of Theorem \ref{tvanquot-t}]
Let $N \lhd G$ be an infinite normal subgroup of $G$ (a finitely generated group). If $N$ is finitely generated as a group, it is possible to pick a Cayley graph $\Gamma = \cay(G,S)$ so that $S$ generates $N$ and $G$. Denote by $\Gamma_N = \cay(N,S\cap N)$. Then $\Gamma$ admits $[G:N]$ (in the present case, infinitely many) copies of $\Gamma_N$ as a subgraph. 

Assume $N$ has polynomial growth at least $d$ (where $d \geq 3$). Let $p \in [1,d/2[$. Let $\Gamma$ and $\Gamma_N$ be as above. Pick the $\Gamma_i$ to be the various copies of $\Gamma_N$ in $\Gamma$: $\Gamma_i$ is the graph restricted to the vertices $N \gamma_i$ where $\gamma_i$ are representatives for the cosets of $N$. Since they are all isomorphic (edges are defined on the left), they have $\IS_d$ with the same constant. Consider two graphs: the graph of the disjoint copies of $\Gamma_N$, say $\Gamma^\Pi$ as above, and this same graph with edges added between the different copies so that the result is connected, call it $\Gamma'$ as above. Note that only finitely many edges need to be added to each copy of $\Gamma_N$ to construct $\Gamma'$ (in a way that $\Gamma'$ is connected) as $G$ is finitely generated and $N$ is normal. 

Assume there is a non-trivial element $g$ in $\D^p(\Gamma)$. By Corollary \ref{tvanspan-c}, it is also non-trivial in $\ssl{\ell^pH}^1(\Gamma')$. By Lemma \ref{tgamn-l}, there are two possibilities. 

First possibility: one of its restriction to some $N \gamma_0$ possesses a non-constant boundary value. But this would imply that $\ssl{\ell^pH}^1(\Gamma_N)$ is non-trivial; this contradicts the hypothesis. 

Second possibility: the restrictions all have constant boundary values, but not the same constant. This would imply that, outside large enough finite sets, these restrictions are arbitrarily close to their constant value. But the gradient of $g$ on the initial graph $\Gamma$ would then not be in $\ell^p(E)$ (in fact, not even in $c_0(E)$) but in $\ell^\infty(E)$. Indeed, since $N$ is normal, the distance from $x \in \Gamma_i$ to $\Gamma_j$ is always the same. This contradicts $g \in \D^p(\Gamma)$.

As before, in the case when $N$ is non-amenable, boundary values belong to the same unreduced class (see Remark \ref{rpans}), so the result will extend to unreduced cohomology in this case.
\end{proof}

\appendix

\section{Ends and degree one reduced $\ell^1$-cohomology}\label{sapp}

\setcounter{teo}{0}
\renewcommand{\theteo}{\thesection.\arabic{teo}}
\renewcommand{\theprop}{\thesection.\arabic{teo}}
\renewcommand{\thecor}{\thesection.\arabic{teo}}

This section is devoted to the reduced degree one $\ell^1$-cohomology. This result is known (even well-known, if one adheres to the rule that at least 3 persons were aware of it): P.~Pansu was aware of this, the main argument of non-vanishing is present in Martin \& Valette \cite[Example 3 in \S{}4]{MV} who mention hearing it from M.~Bourdon. It is included here for the sake of completion.

The ends of a graph are the infinite components of a group which cannot be separated by a finite (\ie compact) set. More precisely, an end $\xi$ is a function from finite sets to infinite connected components of their complement so that $\xi(F) \cap \xi(F') \neq \vide$ (for any $F$ and $F'$). It may also be seen as an equivalence class of (infinite) rays who eventually leave any finite set. Two rays $r$ and $r'$ are equivalent if, for any finite set $F$, the infinite part of $r$ and $r'$ lie in the same (infinite) connected component. 

Thanks to Stallings' theorem, groups with infinitely many ends contain an (non-trivial) amalgamated product or a (non-trivial) HNN extension. Being without ends is equivalent to being finite, and amenable groups may not have infinitely many ends. This may be seen using Stallings' theorem, see also Moon \& Valette \cite{MooV} for a direct proof. An intuitive idea is that a Cayley graph with infinitely many ends contains has a quasi-isometry to a tree $T$ with strong isoperimetric constant, and hence cannot be amenable. Groups with two ends admit $\zz$ as a finite index subgroup. These groups are peculiar, as they have non-trivial reduced $\ell^1$-cohomology in degree $1$, even if their reduced $\ell^p$-cohomology (in all degrees) vanishes for $1<p<\infty$.
 
So outside virtually-$\zz$ groups, all infinite amenable groups have one end. 

\begin{prop}\label{tcohoml1-p}
Let $\Gamma$ be a connected graph, then $\ssl{\ell^1 H}^1(\Gamma) =0$ if and only if the number of ends of $\Gamma$ is $\leq 1$. More precisely, let $\EN = \kk^{\mathrm{ends}(\Gamma)}/ \kk$ be the vector space of functions on ends modulo constants. There is a boundary value map $\beta: \D^1(\Gamma) \to \EN$  such that $\beta(g) = \beta(h) \iff [g] = [h] \in \ssl{\ell^1H}^1(\Gamma)$. 
\end{prop}
Note that the isomorphism is in the category of vector spaces, not of normed vector spaces. In a few cases, the norm on $\EN$ resembles the norm of the quotient $\ell^\infty(|\mathrm{ends}|)/ \kk$. The proof is slightly different than the argument from M.~Bourdon found in Martin \& Valette \cite[Example 3 in \S{}4]{MV}.
\begin{proof}
Note that $\D^1(\Gamma) \subset \ell^\infty(X)$: if $g \in \D^1(\Gamma)$, then, for $P$ a path from $x$ to $y$,
\[
|g(y)| = |g(x) + \sum_{e \in P} g(e) | \leq |g(x)| + \|\nabla g\|_{\ell^1(E)}.
\]
In fact, $\|g \|_{\ell^\infty(X)} \leq \|g\|_{\D^1(\Gamma)} + \inff{x \in X} |g(x)| $. Since functions in $\ell^1$ decrease at $\infty$, if one removes a large enough finite set, the function $g$ on the resulting graph is almost constant. In particular, it is possible to define a value of $g$ on each end: let $B_n$ be the ball of radius $n$ at some fixed vertex (root) $o$, then
\[
\beta g(\xi) :=  \limm{n \to \infty} g(x_n) \text{ where } x_n \in \xi(B_n)
\]
Alternatively, if $r: \zz_{\geq 0} \to X$ is a ray representing the end $\xi$, then the value at $\xi$ can also be defined as $\limm{n \to \infty} g\big(r(n)\big)$. It is fairly straightforward to check these limit do not depend on the choice (of $x_n$ and $o$ or of the ray $r$).

Fix an end $\xi_0$. Then, define $\beta: \D^1(\Gamma) \to \EN$ by changing with a constant the value of $g$ to be $0$ at $\xi_0$ and then looking at the values at the ends. This map is continuous and trivial on $\ell^1(X) + \kk$ (since functions in $\ell^1(X)$ have trivial value at the ends). By continuity, $\srl{\ell^1(X) + \kk}^{\D^1(\Gamma)} \subset \ker \beta$. 

Assume, $\beta(f) =0$, this means that, $\forall \eps>0, \exists X_\eps \subset X$ a finite set such that $|f(X_\eps ^\comp)|<\eps$. The proof then follows \emph{verbatim}, as in Lemma \ref{tbndvalintuit-l}.
\end{proof}
Amenable groups with two ends step strangely out of the crowd: although their $\ell^p$-cohomology is always trivial if $p>1$, it is non-trivial for $p=1$ (actually isomorphic to the base field). An amusing corollary is
\begin{cor}
Let $G$ be a finitely generated group. $G$ has infinitely many ends if and only if for some (and hence all) Cayley graph $\Gamma$, $\forall p \in [1,\infty[, \ssl{\ell^p H}^1(\Gamma) \neq 0$. $G$ has two ends if and only if for some (and hence all) Cayley graph $\Gamma$, $\forall p \in ]1,\infty[, \ssl{\ell^p H}^1(\Gamma) = 0$ but $\ssl{\ell^1 H}^1(\Gamma) = \kk$.
\end{cor}
\begin{proof}
Use Proposition \ref{tcohoml1-p} for reduced $\ell^1$-cohomology, use any vanishing theorem on groups of polynomial growth (Kappos \cite{Kap}, \tcentrinf, or Tessera \cite{Tes}) to get the remaining values of $p$ for groups with two ends, and finally use Theorem \ref{tinclu-t} on groups with infinitely many ends (which are in particular non-amenable).
\end{proof}
It is worth noting that Bekka \& Valette showed in \cite[Lemma 2, p.316]{BekVal} that (for $G$ discrete) the cohomology $H^1(G, \mathbb{C}G)$ is also isomorphic (as a vector space) to $\EN$. Furthermore, by \cite[Proposition 1]{BekVal}, there is an embedding $H^1(G, \mathbb{C}G) \inj \ell^1H^1(G)$. A careful reading would probably reveal this remains injective in reduced cohomology (the only case to check is when $G$ has two ends).

For completion, let us also mention an other extremal case:
\begin{prop}\label{tcohomli-p}
Let $\Gamma$ be an infinite graph, then $\ssl{\ell^\infty H}^1(\Gamma) \neq \{0\}$.
\end{prop}
\begin{proof}
$\ell^\infty H^1(\Gamma)$ is the quotient of Lipschitz functions by bounded Lipschitz functions so is manifestly never trivial. Further, if one takes $g$ to be the distance to a fixed vertex $r$, \ie $g(\gamma) = d(\gamma, r)$, then $[g]$ is not trivial in the reduced cohomology. Indeed, a function with $\| g- h\| < 1/2$ has positive gradient on all edges between the spheres around $r$. As a consequence $h$ may not be bounded, and no element of $\ell^\infty(X)$ may be close to $g$.
\end{proof}
It seems quite plausible that $\ssl{c_0 H}^1(\Gamma) =\{0\}$ for any graph.  Indeed, let $\Omega_\eps$ a big ball such that $\|f\|_{\D^\infty(\Gamma \setminus \Omega_\eps)} \leq \eps$. Start constructing $g_\eps$ by making it equal to $f$ on $\Omega_\eps$. If these functions can be extended so that its gradient is always $< \eps$ outside $\Omega_\eps$ and that it is finitely supported, then $\|f - g_\eps \|_{\D^\infty} < 2\eps$. Hence the class of $f$ would be trivial. It might be useful to use Lemma \ref{tholoso-l} (\ie $f$ may be assumed bounded) to conclude.

\end{document}